\pdfoutput=1
\documentclass[letterpaper]{amsart}
\usepackage[margin=1.25in]{geometry}

 \usepackage{fancyhdr}
\usepackage{amsmath}
\usepackage{amsfonts}
\usepackage{amssymb}
\usepackage{amsthm}

 \usepackage{setspace}

\usepackage{cases}
\usepackage{empheq}

\theoremstyle{plain}
\newtheorem{theorem}{Theorem}[section]

\newtheorem{lemma}[theorem]{Lemma}

\theoremstyle{definition}
\newtheorem{definition}[theorem]{Definition}

\theoremstyle{remark}
\newtheorem{remark}[theorem]{Remark}

\numberwithin{equation}{section} 
\numberwithin{figure}{section}   

\usepackage[square,comma,numbers,sort]{natbib}
\usepackage{url}

\newcommand{\vect}[1]{{#1}}

\newcommand{\ba}{\vect{a}}
\newcommand{\bb}{\vect{b}}

\newcommand{\bD}{\vect{D}}
\newcommand{\bu}{\vect{u}}
\newcommand{\bU}{\vect{U}}
\newcommand{\bv}{\vect{v}}
\newcommand{\bV}{\vect{V}}
\newcommand{\bw}{\vect{w}}

\newcommand{\bx}{\vect{x}}

\newcommand{\bg}{\vect{g}}

\newcommand{\field}[1]{\mathbb{#1}}

\newcommand{\nR}{\field{R}}

\usepackage{mathrsfs}
\DeclareMathAlphabet{\mathpzc}{OT1}{pzc}{m}{it}

\newcommand{\RE}{\text{Re}}

\newcommand{\cnj}[1]{\overline{#1}}

\newcommand{\ip}[2]{\left<#1,#2\right>}

\title[Sensitivity: Navier-Stokes, Data Assimilation]{Sensitivity Analysis for the 2D Navier-Stokes Equations with Applications to Continuous Data Assimilation}

\author{Adam Larios}
\address[Adam Larios]{Department of Mathematics, 
                University of Nebraska--Lincoln,
        Lincoln, NE 68588-0130, USA}
\email[Adam Larios]{alarios@unl.edu}
\author{Elizabeth Carlson}
\address[Elizabeth Carlson]{Department of Mathematics, 
                University of Nebraska--Lincoln,
        Lincoln, NE 68588-0130, USA}
\email[Elizabeth Carlson]{elizabeth.carlson@huskers.unl.edu}

\date{}



\begin{document}

\begin{abstract}
We rigorously prove the well-posedness of the formal sensitivity equations with respect to the Reynolds number corresponding to the 2D incompressible Navier-Stokes equations.  Moreover, we do so by showing a sequence of difference quotients converges to the unique solution of the sensitivity equations for both the 2D Navier-Stokes equations and the related data assimilation equations, which utilize the continuous data assimilation algorithm proposed by Azouani, Olson, and Titi.  As a result, this method of proof provides uniform bounds on difference quotients, demonstrating parameter recovery algorithms that change parameters as the system evolves will not blow-up.  We also note that this appears to be the first such rigorous proof of global existence and uniqueness to strong or weak solutions to the sensitivity equations for the 2D Navier-Stokes equations (in the natural case of zero initial data), and that they can be obtained as a limit of difference quotients with respect to the Reynolds number.
\end{abstract}

\maketitle
\thispagestyle{empty}

\noindent
\section{Introduction}\label{secInt}
\noindent


Turbulent flows are well-known to be chaotic, in the sense that they solutions are highly sensitive to initial conditions (see, e.g., \cite{Constantin_Foias_1985_Lyapunov,Grappin_Leorat_1991JFM}).  However, sensitivity with respect to physical parameters is also an important consideration in terms of making  reliably accurate predictions.  Parameter sensitivity is often measured by formally considering the derivative of a solution with respect to a particular parameter; however, the only rigorous justification of this approach in the literature seems to be limited to linear equations, or non-linear equations under assumptions on the nonlinearity which are too strong to include, e.g., the Navier-Stokes equations of fluids (see, e.g., \cite{Brewer_1982_JMathAnal, Gibson_Clark_1977} for a semigroup theory approach).  Therefore, in the present work, we provide a fully rigorous proof of the global well-posedness of the sensitivity equations for the 2D Navier-Stokes equations.  Specifically, we give a rigorous proof of the existence of unique weak and strong solutions with zero\footnote{Note that considerations of sensitivity arise in the context of perturbations; hence, the natural initial data for a sensitivity equation is identically zero data.} initial data to the associated Reynolds number sensitivity equations specifically for the 2D Navier-Stokes equations.  Moreover, we prove that the derivative of solutions with respect to the velocity is a limit of difference quotients corresponding to different Reynolds numbers. 

We also extend our results to the case of a data assimilation algorithm.  This is because the motivation for this present work arose from our recent work \cite{Carlson_Hudson_Larios_2018}, where an algorithm was proposed to recover an unknown viscosity, or equivalently Reynolds number.  This algorithm works in tandem with a data assimilation method proposed in \cite{Azouani_Titi_2014,Azouani_Olson_Titi_2014}.  This algorithm, commonly referred to as the Azouani-Olson-Titi (AOT) or Continuous Data Assimilation (CDA)  algorithm, has seen much recent work (see, e.g., 
\cite{Albanez_Nussenzveig_Lopes_Titi_2016,
Altaf_Titi_Knio_Zhao_Mc_Cabe_Hoteit_2015,
Bessaih_Olson_Titi_2015,
Biswas_Foias_Mondaini_Titi_2018downscaling,
Biswas_Hudson_Larios_Pei_2017,
Biswas_Martinez_2017,
Carlson_Hudson_Larios_2018,
Carlson_Larios_2020_nonlinConv,
Celik_Olson_Titi_2019,
DiLeoni_Clark_Mazzino_Biferale_2018_unraveling,
DiLeoni_Clark_Mazzino_Biferale_2019,
Desamsetti_Dasari_Langodan_Knio_Hoteit_Titi_2019_WRF,
Farhat_GlattHoltz_Martinez_McQuarrie_Whitehead_2019,
Farhat_Johnston_Jolly_Titi_2018,
Farhat_Jolly_Titi_2015,
Farhat_Lunasin_Titi_2016abridged,
Farhat_Lunasin_Titi_2016benard,
Farhat_Lunasin_Titi_2016_Charney,
Farhat_Lunasin_Titi_2017_Horizontal,
Farhat_Lunasin_Titi_2018_Leray_AOT,
Foias_Mondaini_Titi_2016,
Foyash_Dzholli_Kravchenko_Titi_2014,
GarciaArchilla_Novo_Titi_2018,
Gardner_Larios_Rebholz_Vargun_Zerfas_2020_VVDA,
Gesho_Olson_Titi_2015,
GlattHoltz_Kukavica_Vicol_2014,
Hudson_Jolly_2019,
Ibdah_Mondaini_Titi_2018uniform,
Jolly_Martinez_Olson_Titi_2018_blurred_SQG,
Jolly_Martinez_Titi_2017,
Larios_Pei_2017_KSE_DA_NL,
Larios_Rebholz_Zerfas_2018,
Larios_Victor_2019,
Lunasin_Titi_2015,
Markowich_Titi_Trabelsi_2016_Darcy,
Mondaini_Titi_2018_SIAM_NA,
Pei_2019,
Rebholz_Zerfas_2018_alg_nudge,
Zerfas_Rebholz_Schneier_Iliescu_2019} 
and the references therein.)
Specifically, \cite{Azouani_Olson_Titi_2014} considers the 2D Navier-Stokes system, written abstractly in the form
\begin{align*}
\frac{d\bu}{dt} = F_\RE(\bu).
\end{align*}
The difficulty is that the initial data is unknown; however, it is assumed that the solution can be measured at certain points.  In order to converge to the correct solution, it is proposed to instead consider the system
\begin{empheq}[left=\empheqlbrace]{align*}
 \frac{d\bv}{dt} &= F_\RE(\bv) +\mu(I_h(\bu)-I_h(\bv)) \\
\bv(0) &= \bv_0,
\end{empheq}
where $\mu>0$ is a sufficiently large positive relaxation parameter, $I_h(\bu)$ represents the observational measurements with sufficiently small spacing $h>0$, $\bv_0$ is arbitrarily chosen in a specific Hilbert space, and $F_\RE$ is a nonlinear, nonlocal differential operator depending on the Reynolds number parameter $\RE>0$.  The function $I_h$ is a linear interpolant satisfying certain bounds (see Section \ref{secPre}).  In \cite{Azouani_Olson_Titi_2014}, it was proven that $\bv$ converges to $\bu$ exponentially fast in certain standard norms.  Later, \cite{Carlson_Hudson_Larios_2018} investigated the case of an unknown Reynolds number, gave estimates for the resulting error in the solution, proposed an algorithm to recover the unknown Reynolds number, and demonstrated computationally that the algorithm converges exponentially fast in time to the correct solution.  
However, the algorithm in \cite{Carlson_Hudson_Larios_2018} introduces a discontinuous change in the Reynolds number during the simulation, leading to a desire to ensure that this abrupt change did not lead to the development of, e.g., large shocks in the solution.  Hence, we also prove that the difference quotient methods developed here can be used to prove rigorous results for the sensitivity equations of the modified system of equations via the data assimilation algorithm.
For this system, we prove that the derivative of solutions with respect to the Reynolds number is a well-defined object which is bounded in appropriate function spaces; additionally we prove that the corresponding sensitivity equations are globally well-posed in time in an appropriate sense and that strong solutions are unique.

Sensitivity for partial differential equations has been studied formally in many contexts; see, e.g., 
\cite{
Anderson_Newman_Whitfield_Nielsen_1999_AIAA,
Borggaard_Burns_1997,
Breckling_Neda_Pahlevani_2018_CMA,
Brewer_1982_JMathAnal,
Davis_Pahlevani_2013,
Fernandez_Moubachir_2002_MMMAS,
Gibson_Clark_1977,
Hamby_1994_EnvMonAssess,
Hyoungjin_Chongam_Rho_DongLee_1999_KSIAM,
Kouhi_Houzeaux_Cucchietti_Vazquez_2016_AIAAConf,
Neda_Pahlevani_Rebholz_Waters_2016,
Noacco_Sarrazin_Pianosi_Wagener_2019,
Pahlevani_2004,
Pahlevani_2006,
Rebholz_Zerfas_Zhao_2017_JMFM,
Stanley_Stewart_2002,
Vemuri_Raefsky_1979_IJSS,
Zerfas_Rebholz_Schneier_Iliescu_2019}.  
In \cite{Stanley_Stewart_2002}, it was argued, though only formally, that the sensitivity equations for the steady-state 2D Navier-Stokes equations are globally well-posed.  Some analysis for the sensitivity equations has been carried out in the slightly more general context of a large eddy simulation (LES) model of the 2D Navier-Stokes equations in an unpublished PhD thesis \cite{Pahlevani_2004}, where a formal argument for the global existence and uniqueness of the equations was given, based on formal energy estimates. 

The paper is organized as follows: 
in Section~\ref{secPre}, we describe the mathematical framework for the problems we consider.
In Section~\ref{secSensitivity} we prove the global existence and uniqueness of solutions to the sensitivity equations.  Moreover, we show that these solutions can be realized a limits of difference quotients.
In Section~\ref{secDA}, we extend in the previous section to the context of AOT data assimilation algorithm.  Finally, we summarize our results and implications of this work in Section \ref{secConclusion}.

\section{Preliminaries}\label{secPre}
The statements given in this section without proof for the incompressible Navier-Stokes equations are standard, and proofs can be found in, e.g.,  \cite{Constantin_Foias_1988,Foias_Manley_Rosa_Temam_2001,Robinson_2001,Temam_2001_Th_Num,Temam_1995_Fun_Anal}.  Similarly, equivalent results for the modified data assimilation equations given by the AOT algorithm are stated without proof as well, since proofs were given in \cite{Azouani_Olson_Titi_2014}.  On a general spatial domain $\Omega$, we write the dimensionless incompressible Navier-Stokes equations,
\begin{subequations} \label{NSEpre}
\begin{alignat}{2}
\label{NSE_mo_pre}
\partial_t\bu + (\bu\cdot\nabla)\bu &=-\nabla p + \RE_1^{-1} \triangle\bu + f,
\qquad&& \text{in }\Omega\times[0,T],\\
\label{NSE_div_pre}
\nabla \cdot \bu &=0,
\qquad&& \text{in }\Omega\times[0,T],\\
\label{NSE_IC_pre}
\bu(\bx,0)&=\bu_0(\bx),
\qquad&& \text{in }\Omega.
\end{alignat}
\end{subequations}
where $\RE_1 = \frac{UL}{\nu_1}$ is the dimensionless Reynolds number based on the kinematic viscosity $\nu_1 > 0$, a typical length scale $L$, and typical velocity $U$.
In this paper, we take $\Omega$ to be the torus, i.e. $\Omega = \mathbb{T}^2 = \mathbb{R}^2/\mathbb{Z}^2$, which is an open, bounded, and connected domain with $C^2$ boundary.  We define the space \[\mathcal{V}:= \{f: \Omega \to \mathbb{R}^2 \; | \;f \in \dot{C}_p^\infty(\mathbb{T}^2)\},\] and denote the closures in appropriate spaces by $H := \cnj{\mathcal{V}}$ in $L^2(\Omega; \nR^2)$ and $V := \cnj{\mathcal{V}}$ in $H^1(\Omega; \nR^2)$.  Since $H$ and $V$ are subspaces of $L^2(\Omega; \nR^2)$ and $H^1(\Omega; \nR^2)$, respectively, they are indeed Hilbert spaces which inherit inner products defined by
\begin{alignat}{2}
(\bu,\bv) = \int_{\mathbb{T}^2} \bu \cdot \bv  \; d\bx \qquad &&  ((\bu,\bv))= \sum\limits_{i,j =1}^2 \int_{\mathbb{T}^2} \frac{\partial u_i}{\partial x_j}\frac{\partial v_i}{\partial x_j} \;d\bx, \notag
\end{alignat}
with the obvious norms denoted by $|\bu| = \sqrt{(\bu,\bu)}$ and $\|\bu\| = \sqrt{((\bu,\bu))}$.  Furthermore, due to boundedness of the domain and the mean-zero condition, the following Poincar\'e inequalities hold:
$$\lambda_1\|\bu\|_{L^2}^2\leq\|\nabla\bu\|_{L^2}^2 \text{\quad for\quad} \bu\in V,$$
$$\lambda_1\|\nabla\bu\|_{L^2}^2\leq\|A\bu\|_{L^2}^2 \text{\quad for\quad} \bu\in D(A).$$


We consider the equivalent problem applying the Leray projection to \eqref{NSEpre}, where the Leray projection is defined as $P_\sigma \bu = \bu - \nabla \triangle^{-1} \nabla \cdot \bu$, $P_\sigma : L^2(\Omega) \to H$.  As in \cite{Azouani_Olson_Titi_2014}, we define the Stokes operator $A:\mathcal{D}(A) \to H$, where $\mathcal{D}(A) := \{u\in V: Au\in H\}$ is defined to be the domain of $A$, and the bilinear term $B: V\times V \to V^*$ as the continuous extensions of the operators $A$, defined on $\mathcal{V}$, and $B$, defined on $\mathcal{V} \times \mathcal{V}$,
\begin{alignat}{3}
A\bu = -P_\sigma \triangle \bu & \qquad \text{ and } \qquad & B(u,v) = P_\sigma(\bu \cdot \nabla \bv).\notag
\end{alignat}

We note that, as proven in, e.g., \cite{Constantin_Foias_1988, Robinson_2001, Temam_2001_Th_Num}, $A$ is a linear self-adjoint positive definite operator with a compact inverse.  Hence there exists a complete orthonormal set of eigenfunctions $w_i$ in $H$ such that $Aw_i = \lambda_iw_i$, where the corresponding  eigenvalues are strictly positive and monotonically increasing.  

The bilinear operator, $B$, has the property
\begin{align}\label{bilinear_symmetry}
\ip{B(\bu,\bv)}{\bw} = -\ip{B(\bu,\bw)}{\bv},
\end{align}
for all $\bu,\bv,\bw \in V$, which directly implies that
\begin{align}\label{Borth}
\ip{B(\bu,\bw)}{\bw} = 0,
\end{align}
for all $\bu,\bv,\bw \in V$.  Furthermore, as proven in, e.g., \cite{Constantin_Foias_1988, Robinson_2001, Temam_2001_Th_Num}, we have the following inequalities:
\begin{align}
|\ip{B(\bu,\bv)}{\bw}|&\leq \|\bu\|_{L^\infty(\Omega)}\|\bv\||\bw|  &\text{ for } \bu \in L^\infty(\Omega), \bv \in V, \bw \in H
\label{BINsimple} \\
|\ip{B(\bu,\bv)}{\bw}| &\leq c |\bu|^{1/2}\|\bu\|^{1/2} \|\bv\||\bw|^{1/2}\|\bw\|^{1/2} &\text{ for } \bu,\bv,\bw \in V, \label{BIN} \\
|(B(\bu,\bv),\bw)| &\leq c|\bu|^{1/2}\|\bu\|^{1/2}\|\bv\|^{1/2}|A\bv|^{1/2}|\bw|  &\text{ for } \bu\in V, \bv \in \mathcal{D}(A), \bw \in H \label{BIN2} \\
|(B(\bu,\bv),\bw)| &\leq c |\bu|^{1/2}|A\bu|^{1/2} \|\bv\||\bw| &\text{ for } \bu \in \mathcal{D}(A), \bv \in V, \bw \in H.\end{align}

Due to the periodic boundary conditions, it also holds (in 2D) that
\begin{align}\label{bilinear identity}
(B(\bw,\bw),A\bw)=0 \quad \text{ for every } \quad \bw \in \mathcal{D}(A).
\end{align}
Therefore, for $\bu, \bw \in \mathcal{D}(A)$,
\begin{align}\label{lastbilinear}
(B(\bu,\bw),A\bw)+(B(\bw,\bu),A\bw) = - (B(\bw,\bw),A\bu).
\end{align}

Additionally, further properties of the bilinear term are stated in Lemmas \ref{bilinear_unif_bd} and \ref{bilinear_wk_conv_v*}, which we prove using similar strategies as in \cite{Robinson_2001, Temam_2001_Th_Num}.

\begin{lemma}\label{bilinear_unif_bd}
 Suppose $\{\ba_n\}_{n\in\mathbb{N}}$ and $\{\bb_n\}_{n\in\mathbb{N}}$ are uniformly bounded sequences in $L^2(0,T;V) \cap L^\infty(0,T;H)$. Then $\|B(\ba_n,\bb_n)\|_{L^2(0,T;V^*)}$ is uniformly bounded in $n$.  Moreover, if $\{\ba_n\}_{n\in\mathbb{N}}$ and $\{\bb_n\}_{n\in\mathbb{N}}$ are uniformly bounded in $L^2(0,T;\mathcal{D}(A)) \cap L^\infty(0,T;V)$, then $\|B(\ba_n,\bb_n)\|_{L^2(0,T;H)}$ is uniformly bounded in $n$.
\end{lemma}
\begin{proof}
By the definition of the dual norm and \eqref{bilinear_symmetry}, 
 \begin{align*}
  \|B(\ba_n,\bb_n)\|_{V^*} &= \sup_{\stackrel{\bw \in V}{\|\bw\|=1}} |(B(\ba_n,\bb_n),\bw)|\\
  &= \sup_{\stackrel{\bw \in V}{\|\bw\|=1}} |(B(\ba_n,\bw),\bb_n)|,
 \end{align*}
and applying \eqref{BIN} we obtain
\begin{align*}
  \|B(\ba_n,\bb_n)\|_{V^*} &= \sup_{\stackrel{\bw \in V}{\|\bw\|=1}} |(B(\ba_n,\bw),\bb_n)|\\
  &\leq \sup_{\stackrel{\bw \in V}{\|\bw\|=1}} k |\ba_n|^{1/2}\|\ba_n\|^{1/2}|\bb_n|^{1/2}\|\bb_n\|^{1/2}\|\bw\| \\
  &= k |\ba_n|^{1/2}\|\ba_n\|^{1/2}|\bb_n|^{1/2}\|\bb_n\|^{1/2}.
 \end{align*}
  Using H{\"o}lder's inequality,
  \begin{align*}
   \|B(\ba_n,\bb_n)\|_{L^2(0,T;V^*)}^2 &\leq \int_0^T \|B(\ba_n(s),\bb_n(s)\|_{V^*}^2 ds \\
   &\leq \int_0^T k |\ba_n|\|\ba_n\||\bb_n|\|\bb_n\| ds \\
   &\leq k \|\ba_n\|_{L^\infty(0,T;H)}\|\bb_n\|_{L^\infty(0,T;H)} \int_0^T \|\ba_n(s)\|\|\bb_n(s)\| ds \\
   &\leq k \|\ba_n\|_{L^\infty(0,T;H)}\|\bb_n\|_{L^\infty(0,T;H)}\|\ba_n\|_{L^2(0,T;V)}\|\bb_n\|_{L^2(0,T;V)}.
  \end{align*}
  Hence, since $\{\ba_n\}_{n\in\mathbb{N}}$ and $\{\bb_n\}_{n\in\mathbb{N}}$ are uniformly bounded  in $L^2(0,T;V) \cap L^\infty(0,T;H)$, it follows that 
$\|B(\ba_n,\bb_n)\|_{L^2(0,T;V^*)}$ is uniformly bounded in $n$.

Next, suppose $\{\ba_n\}_{n\in\mathbb{N}}$ and $\{\bb_n\}_{n\in\mathbb{N}}$ are bounded uniformly in $L^2(0,T;\mathcal{D}(A)) \cap L^\infty(0,T;V)$.  Then by definition,
\begin{align*}
\|B(\ba_n,\bb_n)\|_{L^2(0,T;H)}^2 &= \int_0^T |B(\ba_n,\bb_n)|^2 dt \\
&\leq c \int_0^T |\ba_n|\|\ba_n\|\|\bb_n\||A\bb_n| dt \\
&\leq \frac{c}{\lambda_1}\|\ba_n\|^2_{L^\infty(0,T;V)} \int_0^T |A\bb_n|^2 dt\\
&\leq \frac{c}{\lambda_1}\|\ba_n\|^2_{L^\infty(0,T;V)} \int_0^T |A\bb_n|^2 dt\\
&= \frac{c}{\lambda_1}\|\ba_n\|^2_{L^\infty(0,T;V)} \|\bb_n\|_{L^2(0,T;\mathcal{D}(A))}^2,
\end{align*}
which implies $\|B(\ba_n,\bb_n)\|_{L^2(0,T;H)}$ is uniformly bounded in $n$.
\end{proof}

\begin{lemma}\label{bilinear_wk_conv_v*}
Suppose $\{\ba_n\}_{n\in\mathbb{N}}$ and $\{\bb_n\}_{n\in\mathbb{N}}$ are uniformly bounded sequences in $L^2(0,T;V) \cap L^\infty(0,T;H)$.  Furthermore, if the sequences $\{\ba_n\}, \{\bb_n\}$ converge to $\ba, \bb$, respectively, in $L^2(0,T;V)$ weakly and in $L^2(0,T;H)$ strongly, then $B(\ba_n,\bb_n) \stackrel{*}{\rightharpoonup} B(\ba,\bb)$ in $L^2(0,T;V^*)$.
\end{lemma}
\begin{proof}

We need to show that, for each $\bw \in L^2(0,T;V)$, then 
\begin{align*}
\lim\limits_{n \to \infty} \int_0^T \ip{B(\ba_n,\bb_n)}{\bw} dt = \int_0^T \ip{B(\ba,\bb)}{\bw} dt.
\end{align*}
First, take $\bw\in C^1(0,T;C^1(\Omega))$.  Then using the identity \eqref{bilinear identity}, Ladyzhenskaya's, and Poincar{\'e}'s inequality,
\begin{align*}
    \Bigg|\int_0^T \ip{B(\ba_n,\bb_n)}{\bw} &- \ip{B(\ba,\bb)}{\bw} dt\Bigg| = \left|\int_0^T \ip{B(\ba_n,\bw)}{\bb_n} - \ip{B(\ba,\bw)}{\bb} dt \right| \\
    &= \left|\int_0^T (B(\ba_n-\ba,\bw),\bb_n) +  (B(\ba,\bw),\bb_n-\bb) dt \right| \\
    &\leq \int_0^T \|\ba\|_{L^4(\Omega)}\|\bw\|_V \|\bb_n-\bb\|_{L^4(\Omega)} + \|\ba_n-\ba\|_{L^4(\Omega)}\|\bw\|_V \|\bb_n\|_{L^4(\Omega)} dt \\
    &\leq C\lambda_1^{-1/4} \int_0^T \|\ba_n-\ba\|_H^{1/2} \|\ba_n-\ba\|_V^{1/2}\|\bw\|_V\|\bb_n\|_V  \\
    &\hphantom{aaa} + \|\ba\|_V\|\bw\|_V\|\bb_n-\bb\|_H^{1/2}\|\bb_n-\bb\|_V^{1/2}dt \\
    &\leq C\lambda_1^{-1/4} (\|\ba_n-\ba\|_{L^2(0,T;H)}^{1/2} \|\ba_n-\ba\|_{L^2(0,T;V)}^{1/2} \|\bw\|_{L^\infty(0,T;V)} \|\bb_n\|_{L^2(0,T;,V)} \\
    &\hphantom{aaa} + \|\ba\|_{L^2(0,T;V)} \|\bw\|_{L^\infty(0,T;V)} \|\bb_n-\bb\|_{L^2(0,T;H)}^{1/2} \|\bb_n-\bb\|_{L^2(0,T;V)}^{1/2}).
\end{align*}
By the hypotheses, we have that the $\|\ba_n-\ba\|_{L^2(0,T;H)}$ and $\|\bb_n-\bb\|_{L^2(0,T;H)}$ converge to $0$, and all the other terms are bounded (here, we are using that weakly convergence sequences around bounded).  The result now follows, using the density of $C^1(0,T;\mathcal{V})$ in $L^2(0,T;V)$ and the fact that $\|B(\ba_n,\bb_n)\|_{L^2(0,T;V^*)}$ is bounded by Lemma \ref{bilinear_unif_bd}
 \end{proof}
 
\begin{lemma}\label{bilinear_wk_conv}
Suppose $\{\ba_n\}_{n\in\mathbb{N}}$ and $\{\bb_n\}_{n\in\mathbb{N}}$ are uniformly bounded in $L^2(0,T;\mathcal{D}(A)) \cap L^\infty(0,T;V)$.  Furthermore, if $\ba, \bb \in L^2(0,T;V)$, $\ba_n \to \ba$ and $\bb_n \to \bb$ strongly in $L^2(0,T;V)$, and $\{\ba_n\}$ and $\{\bb_n\}$ are bounded above uniformly in $n$ in $L^2(0,T;\mathcal{D}(A))$, then $B(\ba_n,\bb_n) \rightharpoonup B(\ba,\bb)$ in $L^2(0,T;H)$.
\end{lemma}
\begin{proof}
Take $\bw \in C(0,T;H)$; then
\begin{align*}
 &\quad\int_0^T (B(\ba_n,\bb_n),\bw) - (B(\ba,\bb),\bw) dt \\
 &= \int_0^T (B(\ba_n-\ba,\bb_n),\bw) + (B(\ba,\bb_n - \bb),\bw) dt \\
 &\leq \int_0^T |(B(\ba_n-\ba,\bb_n),\bw)| dt + \int_0^T |(B(\ba,\bb_n - \bb),\bw)| dt
\end{align*}
Applying \eqref{BINsimple}, \eqref{BIN2}, Ladyzhenskaya's inequality, and Poincar{\'e}'s inequality we obtain
\begin{align*}
&\quad\int_0^T (B(\ba_n,\bb_n),\bw) - (B(\ba,\bb),\bw) dt \\
 &\leq c\lambda_1^{-1/2} \int_0^T \|\ba_n-\ba\| |A\bb^n| |\bw| dt + c\int_0^T \|\ba\|_{L^\infty(\Omega)}\|\bb_n-\bb\||\bw| dt.
\end{align*}
Applying Agmon's inequality, 
\begin{align*}
&\quad
\int_0^T (B(\ba_n,\bb_n),\bw) - (B(\ba,\bb),\bw) dt \\
 &\leq c\lambda_1^{-1/2} \int_0^T \|\ba_n-\ba\| |A\bb^n| |\bw| dt + c\int_0^T |\ba|^{1/2}|A\ba|^{1/2}\|\bb_n-\bb\||\bw| dt \\
 &\leq c\lambda_1^{-1/2} \int_0^T \|\ba_n-\ba\| |A\bb^n| \|\bw| dt + c\lambda_1^{-1/2}\int_0^T |A\ba|\|\bb_n-\bb\||\bw| dt\\
 &\leq c\lambda_1^{-1/2} \|\ba_n-\ba\|_{L^2(0,T;V)}\|\bw\|_{L^\infty(0,T;H)} \|\bb_n\|_{L^2(0,T;\mathcal{D}(A))} \\
 &\phantom{=} + c\lambda_1^{-1/2} \|\ba\|_{L^2(0,T;\mathcal{D}(A))} \|\bw\|_{L^\infty(0,T;H)}\|\bb_n-\bb\|_{L^2(0,T;V)}
\end{align*}
Since $\ba_n \to \ba$ in $L^2(0,T;V)$, $\bb_n \to \bb$ in $L^2(0,T;V)$, the sequences are bounded above uniformly in $L^2(0,T;\mathcal{D}(A))$, and $\bw$ is continuous in time, then \[\int_0^T (B(\ba_n,\bb_n),\bw) - (B(\ba,\bb),\bw) dt \to 0\]
as $n \to \infty$, and therefore by the density of $C(0,T;H)$ in $L^2(0,T;H)$ and the fact that $\|B(\ba_n,\bb_n)\|_{L^2(0,T;H)}$ is bounded uniformly by Lemma \ref{bilinear_unif_bd}, $B(\ba_n,\bb_n) \rightharpoonup B(\ba,\bb)$ in $L^2(0,T;H)$.
\end{proof}

Finally, without loss of generality, we make the assumption that $f \in L^\infty(0,T;H)$ so that $P_\sigma f = f$.  This allows us to apply $P_\sigma$ to \eqref{NSEpre} to obtain the equivalent set of equations
\begin{subequations} \label{NSE}
\begin{alignat}{2}
\label{NSE_mo} 
\frac{d}{dt}\bu +B(\bu,\bu) &=  \RE_1^{-1} A\bu + f,
\qquad&& \text{in }\Omega\times[0,T],\\
\label{NSE_IC}
\bu(\bx,0)&=\bu_0(\bx),
\qquad&& \text{in }\Omega.
\end{alignat}
\end{subequations}
Using the following corollary of de Rham's theorem \cite{Temam_2001_Th_Num, Foias_Manley_Rosa_Temam_2001}
\begin{align}
\bg = \nabla p \text{ with $p$ a distribution if and only if } \ip{\bg}{h} = 0 \text{ for all h$ \in \mathcal{V}$, }
\end{align}
we can recover the pressure term.

It is well-established that a unique global solution to \eqref{NSE} exists given a force $f$ and initial data $\bu_0$ in appropriate spaces.  For real world applications, it is important to consider the sensitivity of \eqref{NSEpre} to the parameters since it is not necessarily the case we have an exact estimate on the said parameters; however, additional uncertainty in modeling real world systems is introduced by the fact that we do not expect to know  $\bu_0$ exactly, and so cannot compute $\bu(t)$ from $\eqref{NSE}$.  Hence, we will also analyze the sensitivity equations corresponding to a modified system of equations that utilizes measured data collected on the true field $\bu(t)$ over the time interval $[0,T]$.  This modified system of equations incorporates the measured data by introducing a feedback control involving the interpolated data $I_h(\bu(t))$ into \eqref{NSE}, as is done in \cite{Azouani_Olson_Titi_2014},
\begin{subequations}\label{NSEmodified}
\begin{alignat}{2}
\label{NSEmodified_mo}
\frac{d}{dt}\bv + B(\bv,\bv) &= \RE_2^{-1}A\bv + f + \mu P_\sigma(I_h(\bu)-I_h(\bv)) \\
\label{NSEmodified_IC}
\bv(\bx,0)&=\bv_0(\bx).
\end{alignat}
\end{subequations}
Here, $\RE_2 = \frac{UL}{\nu_2}$ with $\nu_2$ a kinematic viscosity approximating $\nu_1$, $\mu$ is a positive relaxation parameter, and $I_h$ is a linear interpolant satisfying 
\begin{align}\label{IntIN}
\|\varphi-I_h(\varphi)\|^2_{L^2(\Omega)} \leq c_0 h^2\|\varphi\|^2_{H^1(\Omega)}
\end{align}

Assuming either no-slip Dirichlet or periodic boundary conditions, \cite{Azouani_Olson_Titi_2014} proved that \eqref{NSEmodified} has a unique solution, stated in the following theorem.
\begin{theorem}
Suppose $I_h$ satisfies \eqref{IntIN} and $\mu c_0h^2 \leq \RE_2^{-1}$, where $c_0$ is the constant from \eqref{IntIN}.  Then the continuous data assimilation equations \eqref{NSEmodified} possess unique strong solutions that satisfy
\begin{alignat}{2}
\bv \in C([0,T];V) \cap L^2((0,T); D(A)) && \text{and } \; \frac{d\bv}{dt} \in L^2((0,T);H),
\end{alignat}
for any $T > 0$.  Furthermore, this solution is in $C([0,T],V)$ and depends continuously on the initial data $\bv_0$ in the $V$ norm.
\end{theorem}

For equations \eqref{NSE} and \eqref{NSEmodified}, we denote the dimensionless Grashof numbers as
\begin{align}
G_1 &= \frac{\RE_1^2}{4\pi^2} \limsup\limits_{t \to \infty} \|f(t)\|_{L^2(\Omega)} \label {G1} \\ 
G_2 &= \frac{\RE_2^2}{4\pi^2} \limsup\limits_{t \to \infty} \|f(t)\|_{L^2(\Omega)} \label{G2}.
\end{align}
Finally, in 2D it is classical that \eqref{NSEpre} possesses a unique global strong solution. 

\section{Sensitivity for 2D Navier-Stokes}\label{secSensitivity}
In this section, we analyze the sensitivity of $\bw$ to the Reynolds number by considering individually the sensitivity of $\bu$ and $\bv$ to the Reynolds number.  We wish to consider taking a derivative of equations \eqref{NSE_mo} and \eqref{NSEmodified_mo} with respect to the Reynolds number.  This has been done formally in many works on sensitivity (see, e.g., \cite{Anderson_Newman_Whitfield_Nielsen_1999_AIAA, Borggaard_Burns_1997, Breckling_Neda_Pahlevani_2018_CMA, Davis_Pahlevani_2013, Fernandez_Moubachir_2002_MMMAS, Hamby_1994_EnvMonAssess, Hyoungjin_Chongam_Rho_DongLee_1999_KSIAM, Kouhi_Houzeaux_Cucchietti_Vazquez_2016_AIAAConf,Pahlevani_2004, Pahlevani_2006, Vemuri_Raefsky_1979_IJSS}), yielding what are known as the \textit{sensitivity equations}.  However, to the best of our knowledge, a rigorous treatment has yet to appear in the literature.  Therefore, we provide a rigorous justification here of the existence and uniqueness of weak and strong solutions to the sensitivity equations in the case of zero initial data, which is the natural data for the sensitivity equation, as discussed below.  Moreover, we prove that these solutions can be realized as limits of difference quotients of Navier-Stokes solutions with respect to different Reynolds numbers.  Indeed, this is the method of our existence proofs, rather than using, e.g., Galerkin methods, fixed-point methods, etc.  Proofs using limits of difference quotients have appeared in the literature before, such as in standard proofs of elliptic regularity, the corresponding result for the Stokes equations, etc.  However, in the present context (i.e., the time-dependent sensitivity equations for 2D Navier-Stokes), we believe such a proof strategy is novel.

Working formally for a moment, we take the derivative of \eqref{NSE_mo} with respect to $\RE$, and denote (again, formally) $\widetilde{\bu} := \frac{d\bu_1}{d(\RE_1^{-1})}$ and $\widetilde{p} := \frac{dp_1}{d(\RE_1^{-1})}$, to obtain
\begin{subequations}\label{NSEpreSens1}
\begin{alignat}{2}
 \widetilde{\bu}_t + \widetilde{\bu}\cdot \nabla \bu_1 + \bu_1\cdot \nabla \widetilde{\bu} - &\RE_1^{-1} \triangle \widetilde{\bu} - \triangle{
 \bu_1} + \nabla \widetilde{p} = 0, \\
 \nabla \cdot \widetilde{\bu} &= 0.
\end{alignat}
\end{subequations}
These are known as the sensitivity equations for the Navier-Stokes equations.  
Similarly, denoting $\widetilde{\bv} := \frac{d\bv_1}{d(\RE_1^{-1})}$ and $\widetilde{q} := \frac{dq_1}{d(\RE_1^{-1})}$, we formally obtain
\begin{subequations}\label{NSEmodifiedSens1}
\begin{alignat}{2}
 \widetilde{\bv}_t + \widetilde{\bv}\cdot \nabla \bv_1 + \bv_1\cdot \nabla \widetilde{\bv} - &\RE_1^{-1} \triangle \widetilde{\bv} - \triangle{
 \bv_1} + \nabla \widetilde{q} = \mu I_h(\widetilde{\bu}-\widetilde{\bv}), \\
 \nabla \cdot \widetilde{\bv} &= 0,
\end{alignat}
\end{subequations}
Below, we prove some well-posedness results for these systems in the case of zero initial data.  We begin by defining what we mean by solutions.

\begin{remark}
We note that the sensitivity equations are a model for the evolution of the instantaneous change in a solution with respect to changes in the (inverse) Reynolds number.  Therefore, the natural initial condition to consider is the case of identically-zero initial data.  Indeed, if the initial data for the sensitivity equations is not identically zero, this would correspond to the case where the initial data for the Navier-Stokes equations depends on the Reynolds number, or equivalently the viscosity, which is not typical of most mathematical treatments of the Navier-Stokes equations.  Thus, although we define weak solutions for general initial data, we only prove their existence for initial data which is identically zero.  Existence for general initial data can be proved using, e.g., Galerkin methods.  However, since our main focus is not on existence, but on showing the solutions can be realized as limits of a (sub)sequence of difference quotients, and moreover since the initial data is naturally taken to be zero in this setting, we use the difference quotient method instead.
\end{remark}

\begin{definition}\label{defweaksolnu} 
Let $T>0$.  Let $\bu\in L^2(0,T;V) \cap C_w(0,T;H)$ be a weak solution to \eqref{NSEpre} with initial data $\bu_0 \in V$ and forcing $f \in L^\infty(0,\infty;V^*)$. A \textit{weak solution} of \eqref{NSEpreSens1} is an element $\widetilde{\bu} \in L^2(0,T;V) \cap C_w(0,T;H)$ satisfying $\frac{d\widetilde{\bu}}{dt} \in L^1_{\text{loc}}(0,T; V^*)$ and
 \begin{align}\label{NSE_sens_weak_form}
  \ip{\widetilde{\bu}_t}{\mathbf{\phi}} + \ip{B(\widetilde{\bu}, \bu)}{ \mathbf{\phi}} + \ip{B(\bu, \widetilde{\bu})}{ \mathbf{\phi}}+\RE_1^{-1}\ip{A\widetilde{\bu}}{\mathbf{\phi}} + \ip{A\bu}{\mathbf{\phi}} = 0
 \end{align}
 for a.e. $t\in[0,T]$, for all $\mathbf{\phi} \in V$, and initial data $\widetilde{\bu}_0 \in H$, satisfied in the sense of $C_w(0,T;H)$. 
 
 If, in addition, $f \in L^\infty(0,\infty;H)$, $\bu_0 \in V$, $\widetilde{\bu}_0 \in V$, and $\bu\in L^2(0,T;V) \cap C_w(0,T;H)$ is a strong solution to \eqref{NSEpre}, then we define a \textit{strong solution} of \eqref{NSEpreSens1} to be a weak solution such that $\widetilde{\bu} \in L^2(0,T;\mathcal{D}(A)) \cap C^0([0,T]; V)$ and $\frac{d\widetilde{\bu}}{dt} \in L^2(0,T; H)$, satisfying \eqref{NSE_sens_weak_form} for a.e. $t \in [0,T]$ and for all $\phi \in H$.
\end{definition}

For the reasons discussed in Remark \ref{remark_no_weak_AOT_solutions} below, we only give a definition of strong solutions for the assimilation equations.

\begin{definition}\label{defweaksolnv} 
Let $T>0$.  Let $\bv$ be a strong solution to \eqref{NSEmodified} with initial data $\bv_0 \in V$ and forcing $f \in L^\infty(0,\infty;H)$.
A \textit{strong solution} of \eqref{NSEmodifiedSens1} is an element $\widetilde{\bv} \in L^2(0,T;\mathcal{D}(A)) \cap C^0([0,T]; V)$ that satisfies 
 \begin{align*}
  \ip{\widetilde{\bv}_t}{\mathbf{\phi}} + \ip{B(\widetilde{\bv}, \bv)}{\mathbf{\phi}} + &\ip{B(\bv, \widetilde{\bv})}{\mathbf{\phi}}+\RE_1^{-1}\ip{A\widetilde{\bv}}{\mathbf{\phi}}\\ &+ \ip{A\bv}{\mathbf{\phi}} = \mu\ip{I_h(\widetilde{\bu}-\widetilde{\bv})}{\mathbf{\phi}}
 \end{align*}
 this equation for a.e. $t \in [0,T]$ and for all $\phi \in H$, where  $\frac{d\widetilde{\bv}}{dt} \in L^2(0,T; H)$ and initial data $\widetilde{\bv}_0 \in V$.
\end{definition}

Before we prove the existence and uniqueness of solutions with zero initial data to these equations, we first consider equations for the difference quotients.  Note that, since these are simple arithmetic operations on the Navier-Stokes equations, the manipulations can be performed rigorously, not just formally.  To this end, let $(\bu_1,p_1)$ be a solution to \eqref{NSEpre} with Reynolds number $\RE_1$ and $(\bu_2,p_2)$ be a solution to \eqref{NSEpre} with Reynolds number $\RE_2$ with the same initial data.  We take the difference of the two versions of \eqref{NSEpre}, each with Reynolds numbers $\RE_1$ and $\RE_2$.  We then divide by the difference in (inverse) Reynolds numbers, yielding the system
\begin{subequations}\label{NSEpreSens}
\begin{alignat}{2}
 \bD_t + \bu_2\cdot \nabla \bD + \bD \cdot \nabla \bu_1 - &\RE_2^{-1} \triangle \bD - \triangle \bu_1 + \nabla P = 0, \\
 \nabla \cdot \bD &= 0, \\
 \bD(\bx,0) &= 0,
\end{alignat}
\end{subequations}
where $\bD = \frac{\bu_1-\bu_2}{\RE_1^{-1}-\RE_2^{-1}}$ and $P := \frac{p_1-p_2}{\RE_1^{-1}-\RE_2^{-1}}$.  As defined, $\bD$ is a strong solution to \eqref{NSEpreSens}, and note that $\bu_1 = (\RE_1^{-1}-\RE_2^{-1}) \bD + \bu_2$.  Additionally, $\bD \in L^2(0,T; \mathcal{D}(A)) \cap C^0([0,T];V)$ and $\bD_t \in L^2(0,T; H)$. However, we need to establish that $\bD$ is the unique solution to \eqref{NSEpreSens}, which is the content of Lemma \ref{D_unique} below. 

\begin{lemma}\label{D_unique}
 Let $T>0$ be given, and let $\bu_1$, $\bu_2\in L^2(0,T; \mathcal{D}(A)) \cap C^0([0,T];V)$ be strong solutions to \eqref{NSEmodified}, with Reynolds numbers $\RE_1$ and $\RE_2$, respectively.  There exists one and only one solution $\bD$ to \eqref{NSEpreSens} that lies in $L^2(0,T; \mathcal{D}(A)) \cap C^0([0,T];V)$, i.e. for all $\phi \in H$, 
 \begin{align*}
(\bD_t,\phi)  + (B(\bD,\bu_1),\phi) + (B(\bu_2,\bD),\phi) + &\RE_2^{-1} (A\bD,\phi) +(A \bu_1,\phi) = 0,
 \end{align*}
 where $\bD_t \in L^2(0,T;H)$.
\end{lemma}

Next, we consider difference quotients for the assimilation system \eqref{NSEmodified}.  Let $(\bv_1,q_1)$ be the solution to \eqref{NSEmodified} with Reynolds number $\RE_1$ and $(\bv_2,q_2)$ be the solution to \eqref{NSEmodified} with Reynolds number $\RE_2$.  Subtracting the two equations and dividing by the difference in the (inverse) Reynolds numbers yields the system \eqref{NSEmodifiedSens},
\begin{subequations}\label{NSEmodifiedSens}
\begin{alignat}{2}
\bD'_t + \bD'\cdot \nabla \bv_1 + \bv_2\cdot \nabla \bD' - &\RE_2^{-1} \triangle \bD' - \triangle \bv_1 + \nabla Q = \mu I_h(\bD-\bD') \\
\nabla \cdot \bD' &= 0\\
\bD'(\bx,0) &= 0,
\end{alignat}
\end{subequations}
where $\bD' := \frac{\bv_1-\bv_2}{\RE_1^{-1}-\RE_2^{-1}}$ and $Q := \frac{q_1-q_2}{\RE_1^{-1}-\RE_2^{-1}}$.  As defined, $\bD'$ is a strong solution to \eqref{NSEmodifiedSens}, and note that $\bv_1 = (\RE_1^{-1}-\RE_2^{-1})\bD' + \bv_2$.  Additionally, $\bD' \in L^2(0,T; \mathcal{D}(A)) \cap C^0([0,T];V)$ and $\bD'_t \in L^2(0,T; H)$.

\begin{lemma}\label{D_prime_unique}
Let $T>0$ be given ,and let $\bv_1$, $\bv_2\in L^2(0,T; \mathcal{D}(A)) \cap C^0([0,T];V)$ be strong solutions to \eqref{NSEmodified}, with Reynolds numbers $\RE_1$ and $\RE_2$, respectively.  There exists a unique strong solution $\bD$ to \eqref{NSEmodifiedSens} that lies in $L^2(0,T; \mathcal{D}(A)) \cap C^0([0,T];V)$, in the sense that for all $\phi \in H$, 
 \begin{align*}
&(\bD_t',\phi) + (B(\bv_2,\bD'),\phi) + (B(\bD',\nabla \bv_1),\phi) + &\RE_2^{-1} (A\bD',\phi) +(A \bv_1,\phi) 
\\&= \mu(P_\sigma I_h(\bD-\bD'),\phi),
 \end{align*}
 where $\bD'_t \in L^2(0,T;H)$.
\end{lemma}

\begin{remark}\label{remark_no_weak_AOT_solutions}
The proofs of the above two lemmata are very similar; hence, we only present the proof of Lemma \ref{D_prime_unique}.  Moreover, we also note that in the case $\mu=0$, the proof of Lemma \ref{D_unique} holds \textit{mutatis mutandis} in the case where $\bu_1$, $\bu_2 \in C^0([0,T];H)\cap L^2(0,T; V)$ are only assumed to be weak solutions to the 2D Navier-Stokes equations, and then one obtains uniqueness of weak solutions to \eqref{NSEpreSens} in the class $C^0([0,T];H)\cap L^2(0,T;V)$.  However, in the case $\mu>0$, the notion of weak solutions for the assimilation equations \eqref{NSEmodified} has not been established in the literature for general interpolants $I_h$, and therefore we assume that the solutions $\bv_1$ and $\bv_2$ are strong solutions to \eqref{NSEmodified}, and prove the uniqueness of strong solutions to \eqref{NSEpreSens}.
\end{remark}

\begin{proof}
 Suppose there exist two solutions $\bD'_1$ and $\bD'_2$.  We consider the difference of the equations
\begin{align}
 \frac{d}{dt} \bD'_1 + B(\bD'_1, \bv_1) + B(\bv_2, \bD'_1) + \RE_2^{-1} A\bD'_1 + A\bv_1 = \mu P_\sigma I_h(\bD-\bD'_1)
\end{align}
and
\begin{align}
 \frac{d}{dt} \bD'_2 + B(\bD'_2, \bv_1) + B(\bv_2, \bD'_2) + \RE_2^{-1} A\bD'_2 + A\bv_1 = \mu P_\sigma I_h(\bD-\bD'_2)
\end{align}
which, defining $\bV := \bD'_1-\bD'_2$, yields
\begin{align}
 \bV_t + B(\bV,\bv_1) + B(\bv_2, \bV) + \RE_2^{-1} A\bV = -\mu P_\sigma I_h(\bV)
\end{align}
with $\bV(0) = 0$.  So, $\bV$ must be a solution to the above equation. Taking the inner product with $\bV$, 
\begin{align}
\frac{1}{2} \frac{d}{dt} |\bV|^2 + b(\bV,\bv_1,\bV) + \RE_2^{-1}\|\bV\|^2 = \ip{-\mu P_\sigma I_h(\bV)}{\bV}
\end{align}
which implies, applying the triangle inequality and Poisson's inequality to the interpolant term as in \cite{Azouani_Olson_Titi_2014},
\begin{align}
 &\quad
 \frac{1}{2} \frac{d}{dt} |\bV|^2 + \RE_2^{-1}\|\bV\|^2 
 \\&\leq\notag 
 c\|\bv_1\||\bV|\|\bV\| + \mu(\sqrt{c_0}h + \lambda_1^{-1})\|\bV\||\bV|
 \\&\leq\notag
 \frac{\mu^2(\sqrt{c_0}h + \lambda_1^{-1})^2}{\RE_2^{-1}} |\bV|^2 + \frac{\RE_2^{-1}}{4} \|\bV\|^2 + \frac{c^2}{2\RE_2^{-1}}\|\bv_1\|^2|\bV|^2 + \frac{\RE_2^{-1}}{2}\|\bV\|^2.
\end{align}
Thus, 
\begin{align}
 \frac{d}{dt} |\bV|^2 &\leq \Big( \frac{\mu^2(\sqrt{c_0}h + \lambda_1^{-1})^2}{\RE_2^{-1}} + \frac{c^2}{2\RE_2^{-1}}\|\bv_1\|^2\Big)|\bV|^2
\end{align}
and Gr{\"o}nwall's inequality implies
\begin{align}
 |\bV(T)|^2 \leq |\bV(0)|^2 \text{exp}\Big(\int_0^T  \frac{\mu^2(\sqrt{c_0}h + \lambda_1^{-1})^2}{\RE_2^{-1}} + \frac{c^2}{2\RE_2^{-1}}\|\bv\|^2 dt \Big).
\end{align}
But $\bV(0) = 0$, and thus $\|\bV\|_{L^\infty(0,T;H)} = 0$ implies that $\bV \equiv 0$.  Hence, solutions to \eqref{NSEmodifiedSens} are unique.
\end{proof}

Since systems \eqref{NSEpreSens} and \eqref{NSEmodifiedSens} have unique strong solutions for every $\RE_2^{-1}>0$, we want to show that, as $\RE_2 \to \RE_1$, the solutions to these equations converge to the unique strong solutions of the respective equations (in the sense of Definitions \ref{defweaksolnu} and \ref{defweaksolnv}) of the formal sensitivity equations \eqref{NSEpreSens1} and \eqref{NSEmodifiedSens1} with initial data $\bu_0\equiv 0$.  We additionally prove that weak solutions exist for the sensitivity equations \eqref{NSEpreSens1} with initial data $\bu_0\equiv 0$.

\begin{theorem}\label{ohhappyday1weak}
Let $\{(\RE_2^{-1})_n\}_{n \in \mathbb{N}}$ be a sequence such that $(\RE_2^{-1})_n \to \RE_1^{-1}$ as $n \to \infty$.  Let
\begin{itemize}
 \item  $\bu$ be a solution to \eqref{NSEpre} with Reynolds number $\RE_1^{-1}$, forcing $f \in L^\infty(0,\infty;H)$, and initial data $\bu_0\equiv\mathbf{0}$;
 \item $\bu_2^n$ solve \eqref{NSEpre} with Reynolds number $(\RE_2^{-1})_n$, forcing $f \in L^\infty(0,\infty;H)$, and initial data $\bu_0 \in V$;
 \item $\{\bD^n\}_{n \in \mathbb{N}}$ be a sequence of strong solutions to \eqref{NSEpreSens} with $\bD^n(0) = 0$.
 \end{itemize}
 Then there is a subsequence of $\{\bD^n\}_{n \in \mathbb{N}}$ that converges in $L^2(0,T;H)$ to a unique weak solution $\bD$ of \eqref{NSEpreSens1} with initial data $\bu_0\equiv 0$ for any $T>0$.
\end{theorem}
\begin{proof}
Let $T>0$ be given. Let $N$ sufficiently large such that for all $n>N$, $\{(\RE_2^{-1})_n\}_{n\in\mathbb{N}} \subset (\frac{\RE_1^{-1}}{2}, \frac{3\RE_1^{-1}}{2})$.  Then, we can follow the proof of strong solutions for \eqref{NSEpre} as in, e.g., \cite{Constantin_Foias_1988, Foias_Manley_Rosa_Temam_2001, Robinson_2001, Temam_2001_Th_Num}, to obtain bounds on $\{\bu_2^n\}$ for $n > N$ in the appropriate spaces that are independent of $(\RE_2^{-1})_n$:
\begin{align*}
\|\bu_2^n\|^2_{L^\infty(0,T;V)} &\leq \|\bu_2^n(0)\|^2 + \frac{ \|f\|^2_{L^2(0,T;H)}}{(\RE_2^{-1})_n} \\
&\leq  \|\bu_0\|^2 + \frac{ 2\|f\|^2_{L^2(0,T;H)}}{\RE_1^{-1}}
\end{align*}
and
\begin{align*}
\|\bu_2^n\|^2_{L^2(0,T;\mathcal{D}(A))} &\leq \frac{1}{(\RE_2^{-1})_n}\|\bu_2^n(0)\|^2 + \frac{ \|f\|^2_{L^2(0,T;H)}}{(\RE_2^{-1})_n^2} \\
&\leq \frac{2}{\RE_1^{-1}}\|\bu_0\|^2 + \frac{ 4\|f\|^2_{L^2(0,T;H)}}{(\RE_1^{-1})^2}.
\end{align*}
Note that $\|f\|^2_{L^2(0,T;H)} < \infty$ since all bounded functions are locally integrable.  Hence there is a subsequence that is relabeled $\bu_2^n \to \bu$ in $L^2(0,T;V)$ for some function $\bu$.  Continuing to follow the proof of strong solutions for \eqref{NSEpre} as in e.g. \cite{Constantin_Foias_1988, Foias_Manley_Rosa_Temam_2001, Robinson_2001, Temam_2001_Th_Num}, we note that $\frac{d \bu_2^n}{dt}$ is uniformly bounded in $n$ in $L^2(0,T;H)$.  Hence, we can find a subsequence which we relabel $\{\bu_2^n\}$ such that
\begin{align*}
 \frac{d \bu_2^n}{dt} \rightharpoonup \frac{d \bu}{dt} \qquad &\text{ in } L^2(0,T;H) \\
(\RE_2^{-1})_n A\bu_2^n \rightharpoonup \RE_1^{-1} A\bu \qquad &\text{ in } L^2(0,T;H) \\
B(\bu_2^n, \bu_2^n) \rightharpoonup B(\bu,\bu) \qquad &\text{ in } L^2(0,T;H).
\end{align*}
Indeed, $\bu$ satisfies \eqref{NSEpre} with corresponding Reynolds number $\RE_1^{-1}$ and thus, by uniqueness and the fact that $\bu_2^n \to \bu$ in $V$, $\bu_1 = \bu$.  Due to Poincar{\'e}'s inequality, we also obtain that $\bu_2^n \to \bu_1$ in $L^2(0,T;H)$.

Let $\bD^n$ be the strong solution to \eqref{NSEpreSens} with $\RE_1^{-1} = (\RE_2^{-1})_n$.  Taking the action of \eqref{NSEpreSens} on $\bD^n$ and using H{\"o}lder's, the bilinear inequalities, and Young's inequality twice, we obtain
\begin{align*}
 \frac{1}{2} \frac{d}{dt} |\bD^n|^2 + (\RE_2^{-1})_n \|\bD^n\|^2 &\leq 
 \frac{c^2}{(\RE_2^{-1})_n} \|\bu_1\|^2 |\bD^n|^2 + \frac{(\RE_2^{-1})_n}{4}\|\bD^n\|^2 \\
 &\phantom{=} + \frac{1}{2(\RE_2^{-1})_n}\|\bu_1\|^2 + \frac{(\RE_2^{-1})_n}{2}\|\bD^n\|^2,
\end{align*}
giving
\begin{align}\label{mainbound_Dn}
 \frac{1}{2}\frac{d}{dt}|\bD^n|^2 + \frac{(\RE_2^{-1})_n}{4}\|\bD^n\|^2 \leq  \frac{c^2}{(\RE_2^{-1})_n} \|\bu_1\|^2 |\bD^n|^2 + \frac{1}{2(\RE_2^{-1})_n}\|\bu_1\|^2.
\end{align}

Dropping the second term on the left hand side, we obtain
\begin{align*}
  \frac{1}{2}\frac{d}{dt}|\bD^n|^2 \leq  \frac{c^2}{(\RE_2^{-1})_n} \|\bu_1\|^2 |\bD^n|^2 + \frac{1}{2(\RE_2^{-1})_n}\|\bu_1\|^2.
\end{align*}

Taking the integral with respect to time on $[0,T]$ and applying Gr{\"o}nwall's inequality, then for a.e. $t \in [0,T]$,
\begin{align*}
 |\bD^n(t)|^2 &\leq \Big[\frac{1}{(\RE_2^{-1})_n}\int_0^T \|\bu_1\|^2 dt \Big] \text{exp}\Big(\int_0^T \frac{2c^2}{(\RE_2^{-1})_n} \|\bu_1\|^2 dt\Big) \\
 &\leq \Big[\frac{2}{\RE_1^{-1}}\int_0^T \|\bu_1\|^2 dt \Big] \text{exp}\Big(\int_0^T \frac{4c^2}{\RE_1^{-1}} \|\bu_1\|^2 dt\Big) =:K_1.
\end{align*}
Since $\bu_1 \in L^2(0,T; V)$, then $\bD^n$ is bounded above uniformly in $L^\infty(0,T;H)$.

Next, refraining from dropping the second term on the left hand side of \eqref{mainbound_Dn}, we estimate
\begin{align*}
 \frac{(\RE_2^{-1})_n}{4} \int_0^T \|\bD^n\|^2 dt &\leq \frac{c^2}{(\RE_2^{-1})_n} \int_0^T\|\bu_1\|^2 |\bD^n|^2dt + \frac{1}{2(\RE_2^{-1})_n}\int_0^T \|\bu_1\|^2dt \\
 &\leq  K_1\frac{c^2}{(\RE_2^{-1})_n} \int_0^T\|\bu_1\|^2 dt + \frac{1}{2(\RE_2^{-1})_n}\int_0^T \|\bu_1\|^2dt
\end{align*}
Rewriting, we obtain
\begin{align*}
 \int_0^T \|\bD^n\|^2 dt  &\leq K_1\frac{4c^2}{(\RE_2^{-1})^2_n} \int_0^T\|\bu_1\|^2 dt + \frac{2}{(\RE_2^{-1})^2_n}\int_0^T \|\bu_1\|^2dt \\
 &\leq K_1\frac{16c^2}{(\RE_1^{-1})^2} \int_0^T\|\bu_1\|^2 dt + \frac{8}{(\RE_1^{-1})^2}\int_0^T \|\bu_1\|^2dt
\end{align*}
Thus, $\bD^n$ is bounded above uniformly in $L^2(0,T;V)$ with respect to $n$.
Hence, by the Banach-Alaoglu Theorem, there exists a subsequence, relabeled as $(\bD^n)$, such that 
\begin{align}\label{BAconv_weak_u}
\bD^n \stackrel{*}{\rightharpoonup} \bD \text{ in } L^\infty(0,T;H)\quad \text{ and }\quad
\bD^n \rightharpoonup \bD \text{ in } L^2(0,T;V). 
\end{align}

Using \eqref{BAconv_weak_u}, note that all uniform bounds in $n$ on the terms in \eqref{NSEpreSens} in $L^2(0,T;V^*)$ are obtained in a similar manner to the proof of weak solutions for \eqref{NSEpre} except for the term $B(\bu_2^n,\bD^n)$.  However, by Lemma \ref{bilinear_unif_bd}, 
\begin{align*}
\|B(\bu_2^n,\bD^n)\|_{L^2(0,T;V^*)} &\leq k \|\bu_2^n\|_{L^\infty(0,T;H)}\|\bD^n\|_{L^\infty(0,T;H)}\|\bu_2^n\|_{L^2(0,T;V)}\|\bD^n\|_{L^2(0,T;V)},
\end{align*}
and due to the following bounds on $\bu_2^n$ (which can be found in \cite{Constantin_Foias_1988,Foias_Manley_Rosa_Temam_2001,Robinson_2001,Temam_2001_Th_Num}, etc.) and the fact that $(\RE_2^{-1})_n \in (\frac{\RE^{-1}}{2},\frac{3\RE^{-1}}{2})$,  
\begin{align*}
 \|\bu_2^n\|^2_{L^\infty(0,T;H)} &\leq |\bu^n_0|^2 + \frac{\|f\|_{L^\infty(0,T;H)}}{\lambda_1^2 (\RE_2^{-1})_n^2} \\
 &\leq |\bu_0|^2 + \frac{4\|f\|_{L^\infty(0,T;H)}}{\lambda_1^2 (\RE_1^{-1})^2}
\end{align*}
and
\begin{align*}
 \|\bu_2^n\|_{L^2(0,T;V)} &\leq \frac{1}{(\RE_2^{-1})_n}|\bu^n(0)|^2 + \frac{\|f\|_{L^\infty(0,T;H)}^2}{\lambda_1(\RE_2^{-1})_n^2} T \\
 &\leq \frac{2}{\RE_1^{-1}}|\bu_0|^2 + \frac{4\|f\|_{L^\infty(0,T;H)}^2}{\lambda_1(\RE_1^{-1})^2} T,
\end{align*}
 and thus $\|B(\bu_2^n,\bD^n)\|_{L^2(0,T;V^*)}$ is bounded above uniformly in $n$ independent of $(\RE_2^{-1})_n$.  Hence, independent of $(\RE_2^{-1})_n$, $d\bD^n/dt$ is bounded above uniformly in $n$ and by the Banach-Alaoglu Theorem a subsequence $\{\bD^n\}_{n\in\mathbb{N}}$ converges weakly to $d\bD/dt$ in $L^2(0,T;V^*)$.  Thus, by the Aubin Compactness Theorem, $\bD^n \to \bD$ strongly in $L^2(0,T;H)$.  Hence, weak continuity in $H$ follows due to the bounds on each of the terms above. Using these facts, we have weak-$*$ convergence in $L^2(0,T;V^*)$ of all but the bilinear terms in the standard sense.  Weak-$*$ convergence of the bilinear terms holds due to Lemma \ref{bilinear_wk_conv_v*},  yielding $B(\bD^n,\bu_1) \stackrel{*}{\rightharpoonup} B(\bD,\bu_1)$ in $L^2(0,T;V^*)$. Additionally since $\bu_2^n \to \bu_1$ strongly in $L^2(0,T;H)$, we can apply Lemma \ref{bilinear_wk_conv_v*} again to obtain that $B(\bu_2^n, \bD^n) \stackrel{*}{\rightharpoonup} B(\bu_1, \bD)$.  Thus, $\widetilde{\bu} := \bD$ satisfies
 \begin{align*}
  \widetilde{\bu}_t + B(\widetilde{\bu},\bu_1) + B(\bu_1, \widetilde{\bu}) + \RE_1^{-1} A\widetilde{\bu} + A\bu_1 = 0
 \end{align*}
in $L^2(0,T;V^*)$.
The initial condition is satisfied by construction.
 To prove uniqueness, suppose that there exist two weak solutions $\widetilde{\bu}_1$ and $\widetilde{\bu}_2$.  We consider the difference of the equations
\begin{align*}
 \frac{d}{dt} \widetilde{\bu}_1 + B(\widetilde{\bu}_1, \bu_1) + B(\bu_1,\widetilde{\bu}_1) + \RE_1^{-1} A\widetilde{\bu}_1 + A \bu_1  = 0
\end{align*}
and
\begin{align*}
 \frac{d}{dt} \widetilde{\bu}_2 + B(\widetilde{\bu}_2, \bu_1) + B(\bu_1, \widetilde{\bu}_2) + \RE_1^{-1} A\widetilde{\bu}_2 + A\bu_1 = 0,
\end{align*}
which, defining $\bU := \widetilde{\bu}_1 - \widetilde{\bu}_2$, yields
\begin{align*}
 \bU_t + B(\bU, \bu_1) + B(\bu_1, \bU) + \RE_1^{-1} A\bU = 0
\end{align*}
with $\bU(0) = 0$.  So, $\bU$ must be a weak solution to the above equation.  Taking the action on $\bU$ and applying the Lions-Magenes Lemma,

\begin{align*}
 \frac{1}{2}\frac{d}{dt} |\bU|^2 + \ip{B(\bU, \bu_1)}{\bU} + \RE_1^{-1}\|\bU\|^2 = 0 
\end{align*}
which implies
\begin{align*}
\frac{1}{2}\frac{d}{dt} |\bU|^2 +  \RE_1^{-1}\|\bU\|^2 &\leq c\|\bU\||\bU|\|\bu_1\| \\
&\leq \frac{c^2}{2\RE^{-1}} \|\bu_1\|^2|\bU|^2 + \frac{\RE_1^{-1}}{2} \|\bU\|^2.
\end{align*}
Dropping the second term, we obtain
\begin{align*}
 \frac{d}{dt} |\bU|^2 \leq \frac{c^2}{2\RE_1^{-1}} \|\bu_1\|^2|\bU|^2,
\end{align*}
and Gr{\"o}nwall's inequality implies that, for a.e. $0 \leq t \leq T$,
\begin{align*}
|\bU(t)|^2 \leq |\bU(0)|^2 \text{exp}\Big(\int_0^T \frac{c^2}{2\RE_1^{-1}} \|\bu_1\|^2 dt\Big). 
\end{align*}
Since we know the $\text{exp}\Big(\int_0^T \frac{c^2}{2\RE_1^{-1}} \|\bu_1\|^2 dt\Big)< \infty$ for all $T>0$ and $\bU(0) = 0$, we have that $\|\bU\|_{L^\infty(0,T; H)} = 0$, which implies that $\bU \equiv 0$.  Hence, weak solutions to \eqref{NSEpreSens} are unique.
\end{proof}

\begin{theorem}\label{ohhappyday1}
 Let $\{(\RE_2^{-1})_n\}_{n \in \mathbb{N}}$ be a sequence such that $(\RE_2^{-1})_n \to \RE_1^{-1}$ as $n \to \infty$.  Let
\begin{itemize}
 \item  $\bu$ be the solution to \eqref{NSEpre} with Reynolds number $\RE_1^{-1}$, forcing $f\in L^\infty(0,\infty;H)$, and initial data $\bu_0$;
 \item $\bu_2^n$ solve \eqref{NSEpre} with Reynolds number $(\RE_2^{-1})_n$, forcing $f \in L^\infty(0,\infty;H)$, and initial data $\bu_0 \in V$
 \item $\{\bD^n\}_{n \in \mathbb{N}}$ be a sequence of strong solutions to \eqref{NSEpreSens} with $\bD^n(0) = 0$.
 \end{itemize}
 Then there is a subsequence of $\{\bD^n\}_{n \in \mathbb{N}}$ that converges in $L^2(0,T;V)$ to a unique strong solution $\bD$ of \eqref{NSEpreSens1} with initial data $\bu_0\equiv 0$.
\end{theorem}
\begin{proof}

Let $T>0$ be given, and let $N>0$ be large enough that $n>N$ implies $\{(\RE_2^{-1})_n\} \subset (\frac{\RE_1^{-1}}{2}, \frac{3\RE_1^{-1}}{2})$.  Then by the argument in Theorem \ref{ohhappyday1weak}, we can obtain a subsequence which we relabel $\{\bu_2^n\}$ such that $\bu_2^n \to \bu_1$ in $L^2(0,T;V)$.

Consider $\bD^n$ to be the strong solution to \eqref{NSEpreSens} with Reynolds number $(\RE_2^{-1})_n$.  Taking a justified inner product of \eqref{NSEpreSens} with $A\bD^n$,

\begin{align*}
\frac{1}{2} \frac{d}{dt} \|\bD^n\|^2 + (\RE_2^{-1})_n |A\bD^n|^2 &= -(B(\bD^n,\bu_1),A\bD^n) - (B(\bu_2^n, \bD^n), A\bD^n) \\
&\phantom{=}- (A\bu_1,A \bD^n).
\end{align*}
Applying Young's inequality, we obtain
\begin{align*}
\frac{1}{2} \frac{d}{dt} \|\bD^n\|^2 + \frac{(\RE_2^{-1})_n}{2} |A\bD^n|^2 &\leq -(B(\bD^n,\bu_1),A\bD^n) - (B(\bu_2^n, \bD^n), A\bD^n) \\
&\phantom{=} + \frac{1}{2(\RE_2^{-1})_n}|A\bu_1|^2.
\end{align*}
Applying \eqref{BINsimple} to the second bilinear term,
\begin{align*}
\frac{1}{2} \frac{d}{dt} \|\bD^n\|^2 + \frac{(\RE_2^{-1})_n}{2} |A\bD^n|^2 &\leq -(B(\bD^n,\bu_1),A\bD^n) + \|\bu_2^n\|_{L^\infty(\Omega)}\|\bD^n\||A\bD^n| \\
&\phantom{=} + \frac{1}{2(\RE_2^{-1})_n}|A\bu_1|^2 \\
&\leq \frac{2k^2}{(\RE_2^{-1})_n} |\bu_2^n||A\bu_2^n|\|\bD^n\|^2 + \frac{(\RE_2^{-1})_n}{8} |A\bD^n|^2 \\
&\phantom{=} -(B(\bD^n,\bu_1),A\bD^n)  + \frac{1}{2(\RE_2^{-1})_n}|A\bu_1|^2
\end{align*}
and applying \eqref{BIN2} to the first bilinear term,
\begin{align*}
&
\frac{1}{2} \frac{d}{dt} \|\bD^n\|^2 + \frac{3(\RE_2^{-1})_n}{8} |A\bD^n|^2 
\\\leq&
\frac{2k^2}{(\RE_2^{-1})_n}|\bu_2^n||A\bu_2^n|\|\bD^n\|^2 \\
&\phantom{=} + c|\bD^n|^{1/2}\|\bD^n\|^{1/2}\|\bu_1\|^{1/2}|A\bu_1|^{1/2}|A\bD^n| + \frac{1}{2(\RE_2^{-1})_n}|A\bu_1|^2 \\
\leq& \frac{2k^2}{(\RE_2^{-1})_n}|\bu_2^n||A\bu_2^n|\|\bD^n\|^2  + \frac{2c^2}{\lambda_1(\RE_2^{-1})_n} \|\bD^n\|^2\|\bu_1\||A\bu_1|\\
&\phantom{=}+ \frac{(\RE_2^{-1})_n}{8}|A\bD^n|^2  + \frac{1}{2(\RE_2^{-1})_n}|A\bu_1|^2
\end{align*}
which can be rewritten as 
\begin{align*}
&
 \frac{d}{dt} \|\bD^n\|^2 + \frac{(\RE_2^{-1})_n}{2} |A\bD^n|^2 
 \\\leq& 
 \Big(\frac{4k^2}{(\RE_2^{-1})_n}|\bu_2^n||A\bu_2^n| +  \frac{4c^2}{\lambda_1(\RE_2^{-1})_n}\|\bu_1\||A\bu_1| \Big) \|\bD^n\|^2  + \frac{1}{(\RE_2^{-1})_n}|A\bu_1|^2.
\end{align*}
Integrating on both sides in time, with $0 \leq t \leq T$,
\begin{align*}
\|\bD^n(t)\|^2 + &\frac{(\RE_2^{-1})_n}{2} \int_0^t |A\bD^n|^2 ds \leq  \frac{1}{(\RE_2^{-1})_n}\int_0^t|A\bu_1(s)|^2 ds \\
 &\phantom{=} + \int_0^t\Big(\frac{4k^2}{(\RE_2^{-1})_n} |\bu_2^n(s)||A\bu_2^n(s)| +  \frac{4c^2}{\lambda_1(\RE_2^{-1})_n}\|\bu_1(s)\||A\bu_1(s)| \Big) \|\bD^n(s)\|^2 ds 
\end{align*}
Dropping the second term on the left hand side, we apply Gr{\"o}nwall's inequality to obtain
\begin{align*}
\|\bD^n(t)\|^2 &\leq \alpha_n(t) \;\text{exp} \Big(\int_0^t\frac{4k^2}{(\RE_2^{-1})_n} |\bu_2^n(s)||A\bu_2^n(s)| +  \frac{4c^2}{\lambda_1(\RE_2^{-1})_n}\|\bu_1(s)\||A\bu_1(s)| ds\Big) \\
&\leq \alpha(t)\;\text{exp} \Big(\int_0^t\frac{8k^2}{\RE_1^{-1}} |\bu_2^n(s)||A\bu_2^n(s)| +  \frac{8c^2}{\lambda_1\RE_1^{-1}}\|\bu_1(s)\||A\bu_1(s)| ds\Big).
\end{align*}
where $\alpha_n(t) := \frac{2}{(\RE_2^{-1})_n}\int_0^t|A\bu_1(s)|^2 ds \leq \alpha(t) := \frac{4}{\RE_1^{-1}}\int_0^t|A\bu_1(s)|^2 ds$.  Since 
\[ \int_0^T |A\bu_2^n|^2 ds \leq \|\bu_0\|^2 + \frac{\|f\|_{L^2(0,T; H)}}{(\RE_2^{-1})_n}\]
as proven in, e.g., \cite{Constantin_Foias_1988, Robinson_2001, Foias_Manley_Rosa_Temam_2001, Temam_2001_Th_Num}, then
\begin{align*}
 \sup\limits_{t\in [0,T]} \|\bD^n(t)\|^2 &\leq \alpha(T) \frac{8k^2}{\lambda_1^2(\RE_1^{-1})} \int_0^T |A\bu_2^n|^2 ds + \frac{8c^2}{\lambda_1(\RE_1^{-1})}\|\bu_1(s)\||A\bu_1(s)| ds\\
 &\leq \alpha(T) \frac{8k^2}{\lambda_1^2(\RE_1^{-1})}\Big[ \|\bu_0\|^2 + \frac{\|f\|_{L^2(0,T; H)}}{(\RE_2^{-1})_n}\Big] \\
 &\phantom{=} + \alpha(T) \int_0^T \frac{8c^2}{\lambda_1(\RE_1^{-1})}\|\bu_1(s)\||A\bu_1(s)| ds\\
 &\leq \alpha(T) \frac{8k^2}{\lambda_1^2(\RE_1^{-1})}\Big[ \|\bu_0\|^2 + \frac{2\|f\|_{L^2(0,T; H)}}{(\RE_1^{-1})}\Big] \\
 &\phantom{=} + \alpha(T) \int_0^T \frac{8c^2}{\lambda_1(\RE_1^{-1})}\|\bu_1(s)\||A\bu_1(s)| ds\\
\end{align*}

This implies that $\bD^n \in L^\infty(0,T; V)$ and $\{\bD^n\}$ is uniformly bounded in this space.

Additionally, considering again the inequality
\begin{align*}
\|\bD^n(t)\|^2 + &\frac{(\RE_2^{-1})_n}{2} \int_0^t |A\bD^n|^2 ds \leq \frac{1}{(\RE_2^{-1})_n}\int_0^t|A\bu_1(s)|^2 ds \\
 &\phantom{=} + \int_0^t\Big(\frac{4k^2}{(\RE_2^{-1})_n} |\bu_2^n(s)||A\bu_2^n(s)| +  \frac{4c^2}{\lambda_1(\RE_2^{-1})_n}\|\bu_1(s)\||A\bu(s)| \Big) \|\bD^n(s)\|^2 ds.
\end{align*}
we set $t = T$, drop the first term on the left hand side, and bound the Reynolds number above to obtain
\begin{align*}
\int_0^T |A\bD^n|^2 ds &\leq \frac{8}{(\RE_1^{-1})^2}\left(\int_0^T|A\bu_1(s)|^2 ds \right)\\
 &\phantom{=} + \int_0^T\Big(\frac{32k^2}{\lambda_1(\RE_1^{-1})^2} |A\bu_2^n(s)|^2 +  \frac{32c^2}{\lambda_1(\RE_1^{-1})^2}\|\bu_1(s)\||A\bu(s)| \Big) \|\bD^n(s)\|^2 ds
\end{align*}
By the fact that $\{\|\bu_2^n\|_{L^2(0,T;\mathcal{D}(A))}\}$ is bounded above in $n$ as demonstrated in Theorem \ref{ohhappyday1weak} and the result that $\{\|\bD^n\|_{L^\infty(0,T;V)}\}$ is bounded above uniformly in $n$, we also have that $\{\|\bD^n\|_{L^2(0,T; \mathcal{D}(A))}\}$ is bounded above uniformly in $n$.  Since $\{\bD^n\}$ is bounded above uniformly in $n$ in both $L^\infty(0,T;V)$ and $L^2(0,T;\mathcal{D}(A))$, then we can conclude that there exists a subsequence, which we relabel as $\{\bD^n\}$, such that 
\begin{equation}\label{BAconv_str_u}
\bD^n \stackrel{*}{\rightharpoonup} \bD \text{ in } L^\infty(0,T; V) \text{ and } \bD^n \rightharpoonup \bD \text{ in } L^2(0,T;\mathcal{D}(A)). 
\end{equation}
Using \eqref{BAconv_str_u}, note that all uniform bounds in $n$ on the terms in \eqref{NSEpreSens} in $L^2(0,T;H)$ are obtained in a similar manner to the proof of strong solutions for the \eqref{NSEpre} and are independent of $(\RE_2^{-1})_n$ except for the bilinear terms.  The bilinear terms are bounded uniformly in $L^2(0,T;H)$ with respect to $n$, due to Lemma \ref{bilinear_unif_bd}.  Hence, $\frac{d\bD^n}{dt}$ is bounded above uniformly in $n$ in $L^2(0,T;H)$.  Thus, as in, e.g., \cite{Robinson_2001, Constantin_Foias_1988, Foias_Manley_Rosa_Temam_2001, Temam_2001_Th_Num}, 
\[\frac{d\bD^n}{dt} \rightharpoonup \frac{d\bD}{dt} \text{ in } L^2(0,T;H).\]
Hence, by the Aubin Compactness Theorem, $\bD^n \to \bD$ strongly in $L^2(0,T;V)$.  As in, e.g., \cite{Robinson_2001, Temam_2001_Th_Num, Foias_Manley_Rosa_Temam_2001, Constantin_Foias_1988}, $\bD \in C^0(0,T;V)$. Using these facts, we have weak convergence in $L^2(0,T;H)$ for all except the bilinear terms in the standard sense.  Weak convergence of the bilinear terms holds due to Lemma \ref{bilinear_wk_conv}.  Hence, $\widetilde{\bu} := \bD$ satisfies
 \begin{align*}
  \widetilde{\bu}_t + B(\widetilde{\bu},\bu_1) + B(\bu_1, \widetilde{\bu}) + \RE_1^{-1} A\widetilde{\bu} + A\bu_1 = 0
 \end{align*}
in $L^2(0,T;H)$.

The initial condition is also satisfied by construction.
Uniqueness holds due to the results in Theorem \ref{ohhappyday1weak}.
\end{proof}

\section{Extension to a Data Assimilation Algorithm}\label{secDA}
In this section, we extend our analysis above to the context of a data assimilation algorithm, as discussed in the introduction.
\begin{theorem}\label{ohhappyday2}
Let $\{(\RE_2^{-1})_n\}_{n \in \mathbb{N}}$ be a sequence such that $(\RE_2^{-1})_n \to \RE_1^{-1}$ as $n \to \infty$.  Choose $\mu$ and $h$ such that $4\mu c_0 h^2 \leq (\RE_2^{-1})_n \leq \frac{3\RE_1^{-1}}{2}$.  Let
\begin{itemize}
 \item  $\bv$ be the solution to \eqref{NSEmodified} with Reynolds number $\RE_1^{-1}$, forcing $f\in L^\infty(0,\infty;H)$, and initial data $\bv_0$;
 \item $\bv_2^n$ solve \eqref{NSEmodified} with Reynolds number $(\RE_2^{-1})_n$, forcing $f\in L^\infty(0,\infty;H)$, and initial data $\bv_0 \in V$;
 \item $\{\bD^{n'}\}_{n \in \mathbb{N}}$ be a sequence of strong solutions to \eqref{NSEmodifiedSens} with $\bD^{n'}(0) = 0$.
 \end{itemize}
 Then there is a subsequence of $\{\bD^{n'}\}_{n \in \mathbb{N}}$ that converges in $L^2(0,T;V)$ to a unique solution $\bD'$ of \eqref{NSEmodifiedSens1} with initial data $\bu_0\equiv 0$.
\end{theorem}
\begin{proof}

 Let $T > 0$. Note that since $\{(\RE_2^{-1})_n\} \subset (\frac{\RE_1^{-1}}{2}, \frac{3\RE_1^{-1}}{2})$  for $n > N$ for some sufficiently large $N$, we can follow the proof of strong solutions for \eqref{NSEmodified} in \cite{Azouani_Olson_Titi_2014} to obtain bounds on $\{\bv_2^n\}_{n>N}$ in the appropriate spaces that are independent of $(\RE_2^{-1})_n$.  First, we note that \cite{Azouani_Olson_Titi_2014} quickly proves $|f+\mu P_\sigma I_h(\bu_2^n)|^2 \leq M_n$ since $|P_\sigma I_h(\bu_2^n)|^2 \leq |\bu_2^n|^2$.  However, since $\bu_2^n$ is bounded above uniformly in $n$ (see the proof of Theorem \ref{ohhappyday1weak}), we have that $|f+\mu P_\sigma I_h(\bu_2^n)|^2 \leq m$ for some $m$ independent of $n$.  Thus, we have the following bounds from \cite{Azouani_Olson_Titi_2014} bounded above uniformly in $n$:
\begin{align}\label{vbound_2_inf}
\|\bv_2^n\|_{L^\infty(0,T;H)}^2 &\leq |\bv^n_2(0)|^2 + \frac{m}{\mu (\RE_2^{-1})_n \lambda_1} \\
   &\leq |\bv_0|^2 + \frac{2m}{\mu \RE_1^{-1} \lambda_1} \notag,
\end{align}
\begin{align}\label{vbound_2_2}
\|\bv_2^n\|_{L^2(0,T;V)}^2 &\leq \frac{1}{(\RE_2^{-1})_n}|\bv^n_2(0)|^2 + \frac{T}{\mu(\RE_2^{-1})_n} m \\
 &\leq \frac{2}{\RE_1^{-1}}|\bv_0|^2 + \frac{2T}{\mu\RE_1^{-1}} m, \notag
\end{align}
\begin{align}\label{vbound_h1_inf}
 \|\bv_2^n\|_{L^\infty(0,T;V)}^2 &\leq \frac{1}{\psi(T)} \Big[ \|\bv^n_2(0)\|^2 + \frac{4T}{(\RE_2^{-1})_n}m\Big] \\
 &\leq \frac{1}{\overline{\psi(T)}} \Big[ \|\bv_0\|^2 + \frac{8T}{\RE_1^{-1}}m\Big]\notag
\end{align}
where 
\begin{align*}
\frac{1}{\psi(T)} &= \text{exp}\Big\{\frac{c}{(\RE_2^{-1})^3_n}\int_0^T|\bv_2^n|^2\|\bv_2^n\|^2 ds\Big\} \\
&\leq \frac{1}{\overline{\psi(T)}} = \text{exp}\Big\{\frac{8c}{(\RE_1^{-1})^3}\int_0^T|\bv_2^n|^2\|\bv_2^n\|^2 ds\Big\},
\end{align*}
which is bounded above uniformly in $n$ due to \eqref{vbound_2_inf} and \eqref{vbound_2_2}, and
\begin{align*}
&
\|\bv_2^n\|_{L^2(0,T;\mathcal{D}(A))}^2
\\\leq&
\frac{1}{(\RE_2^{-1})_n}\|\bv^n_2(0)\|^2 + \frac{c}{(\RE_2^{-1})^3_n} \int_0^T (|\bv_2^n|^2\|\bv_2^n\|^4 + \frac{4}{(\RE_2^{-1})_n}|f + P_\sigma I_h(\bu_2^n)|^2) ds 
\\\leq&
\frac{2}{(\RE_1)^{-1}}\|\bv_0\|^2 + \frac{8c}{(\RE_1^{-1})^3} \int_0^T |\bv_2^n|^2\|\bv_2^n\|^4 ds + \frac{8T}{\RE_1^{-1}}m,
\end{align*}
which is bounded above uniformly in $n$ due to \eqref{vbound_2_inf}, \eqref{vbound_2_2}, \eqref{vbound_h1_inf}.  Hence, we will obtain a subsequence that is relabeled $\bv_2^n \to \bv$ in $L^2(0,T;V)$ for some function $\bv$.  Indeed, we see that by identical arguments presented in Theorem \ref{ohhappyday1weak}, $\bv = \bv_1$.  Also due to Poincar{\'e}'s inequality, we obtain that $\bv_2^n \to \bv_1$ in $L^2(0,T;H)$.

Let $\{\bD^{n'}\}_{n \in \mathbb{N}}$ be a sequence of solutions to \eqref{NSEmodifiedSens}.  We consider the Leray projection of \eqref{NSEmodifiedSens}:
\begin{align*}
 \frac{d}{dt}\bD^{n'} + B(\bD^{n'},\bv_1) + B(\bv_2^n, \bD^{n'}) + (\RE_2^{-1})_nA \bD^{n'} + A \bv_1 = \mu P_\sigma I_h(\bD^n-\bD^{n'}).
\end{align*}
The existence proof for \eqref{NSEmodifiedSens1} closely follows the proof of Theorem \ref{ohhappyday1}, with some modifications on the bounds of $\bD^{n'}$ which we show below.  
Taking the inner product with $A\bD^{n'}$ and proceeding as in the proof of Theorem \ref{ohhappyday1}, we obtain
\begin{align}\label{dnprime_inequality}
 \frac{1}{2}\frac{d}{dt} \|\bD^{n'}\|^2 + \frac{(\RE_2^{-1})_n}{4} |A\bD^{n'}|^2 &\leq \Big(\frac{2k^2}{(\RE_2^{-1})_n}|\bv_2^n||A\bv_2^n| +  \frac{2c^2}{\lambda_1(\RE_2^{-1})_n}\|\bv_1\||A\bv_1| \Big) \|\bD^{n'}\|^2 
 \\\notag
 &\phantom{=} + \frac{1}{2(\RE_2^{-1})_n}|A\bv_1|^2 + \mu (I_h(\bD^n - \bD^{n'}),A\bD^{n'}).
\end{align}
We slightly modify the inequalities obtained in \cite{Azouani_Olson_Titi_2014} for the interpolant term,
\begin{align*}
\mu|(I_h(\bD^{n'}), A\bD^{n'})| &\leq \frac{4\mu^2}{(\RE_2^{-1})_n}|\bD^{n'}-I_h(\bD^{n'})|^2 + \frac{(\RE_2^{-1})_n}{16}|A\bD^{n'}|^2 - \mu\|\bD^{n'}\|^2 \\
&\leq \frac{4\mu^2c_0h^2}{(\RE_2^{-1})_n}\|\bD^{n'}\|^2 + \frac{(\RE_2^{-1})_n}{16}|A\bD^{n'}|^2 - \mu\|\bD^{n'}\|^2 \\
&\leq \frac{(\RE_2^{-1})_n}{16} |A\bD^{n'}|^2.
\end{align*}
Also,
\begin{align*}
\mu|(I_h(\bD^n), A\bD^{n'})| \leq \frac{4\mu^2}{(\RE_2^{-1})_n} |\bD^n|^2 + \frac{(\RE_2^{-1})_n}{16}|A\bD^{n'}|^2.
\end{align*}
Using these inequalities in \eqref{dnprime_inequality}:
\begin{align*}
 \frac{1}{2}\frac{d}{dt} \|\bD^{n'}\|^2 + \frac{(\RE_2^{-1})_n}{8} |A\bD^{n'}|^2 &\leq \Big(\frac{2k^2}{(\RE_2^{-1})_n}|\bv_2^n||A\bv_2^n| +  \frac{2c^2}{\lambda_1(\RE_2^{-1})_n}\|\bv_1\||A\bv_1| \Big) \|\bD^{n'}\|^2 \\
 &\phantom{=} + \frac{1}{2(\RE_2^{-1})_n}|A\bv_1|^2 + \frac{4}{(\RE_2^{-1})_n}|\bD^n|^2  \\
 &\leq \Big(\frac{2k^2}{(\RE_2^{-1})_n}|\bv_2^n||A\bv_2^n| +  \frac{2c^2}{\lambda_1(\RE_2^{-1})_n}\|\bv_1\||A\bv_1| \Big) \|\bD^{n'}\|^2 \\
 &\phantom{=} + \frac{1}{2(\RE_2^{-1})_n}|A\bv_1|^2 + \frac{4}{\lambda_1^2(\RE_2^{-1})_n}|A\bD^n|^2.
\end{align*}
Following identical arguments as in Theorem \ref{ohhappyday1} with \[\alpha_n(t) := \frac{1}{2(\RE_2^{-1})_n}|A\bv_1|^2 + \frac{4}{\lambda_1^2(\RE_2^{-1})_n}|A\bD^n|^2 \leq \alpha(t) := \frac{1}{\RE_1^{-1}}|A\bv_1|^2 + \frac{8}{\lambda_1^2\RE_1^{-1}}|A\bD^n|^2,\] along with the fact that $P_\sigma I_h(\bD^n-\bD^{n'})$ is bounded uniformly in $n$ in $L^2(0,T;H)$, we obtain a subsequence relabeled $\bD^{n'} \to \bD'$ in $L^2(0,T;V)$. Indeed, let $\phi \in L^2(0,T;H)$; then
\begin{align*}
&\quad
\int_0^T (P_\sigma I_h(\bD^n-\bD^{n'}) - P_\sigma I_h(\bD-\bD'),\phi) ds 
\\&\leq 
\int_0^T |I_h(\bD^n-\bD^{n'}) - I_h(\bD-\bD')||\phi| ds \\
&\leq \int_0^T |I_h(\bD^n-\bD) - I_h(\bD^{n'}-\bD')||\phi| ds \\
&\leq \int_0^T |[(\bD^n-\bD) - (\bD^{n'}-\bD')] - I_h((\bD^n-\bD)-(\bD^{n'}-\bD'))||\phi| ds \\
&\phantom{=}+ \int_0^T |(\bD^n-\bD) - (\bD^{n'}-\bD')||\phi| ds \\
&\leq \sqrt{c_0}h \int_0^T \|(\bD^n-\bD) - (\bD^{n'}-\bD')\||\phi| ds \\
&\phantom{=}+ \frac{1}{\lambda_1^{1/2}} \int_0^T \|(\bD^n-\bD) - (\bD^{n'}-\bD')\||\phi| ds \\
&\leq (\sqrt{c_0}h \|\bD^n-\bD\|_{L^2(0,T;V)}\|\phi\|_{L^2(0,T;H)} \\
&\phantom{=} + \frac{1}{\lambda_1^{1/2}} \|\bD^{n'}-\bD'\|_{L^2(0,T;V)}\|\phi\|_{L^2(0,T;H)}).
\end{align*}

Additionally, since we now have that $\bD^{n'} \to \bD$ in $L^2(0,T;V)$, then $P_\sigma I_h(\bD^n-\bD^{n'}) \rightharpoonup P_\sigma I_h(\bD-\bD')$ in $L^2(0,T;H)$ and we conclude $\bD'$ is a strong solution in the sense of Definition \ref{defweaksolnv}. 

To show that the solutions are unique, we consider the 
difference of the equations
\begin{align*}
 \frac{d}{dt} \widetilde{\bv}_1 + B(\widetilde{\bv}_1, \bv_1) + B(\bv_1, \widetilde{\bv}_1) + \RE^{-1} A\widetilde{\bv}_1 + A\bv = \mu P_\sigma I_h(\widetilde{\bu}-\widetilde{\bv}_1)
\end{align*}
and
\begin{align*}
 \frac{d}{dt} \widetilde{\bv}_2 + B(\widetilde{\bv}_2, \bv_1) + B(\bv_1, \widetilde{\bv}_2) + \RE^{-1} A\widetilde{\bv}_2 + A\bv_1 = \mu P_\sigma I_h(\widetilde{\bu}-\widetilde{\bv}_2)
\end{align*}
which, defining $\bV := \widetilde{\bv}_1-\widetilde{\bv}_2$, yields
\begin{align*}
 \bV_t + B(\bV,\bv_1) + B(\bv_1, \bV) + \RE^{-1} A\bV = -\mu P_\sigma I_h(\bV)
\end{align*}
with $\bV(0) = 0$.  So, $\bV$ must be a solution to the above equation. Taking the action on $\bV$ and applying the Lions-Magenes Lemma, 
\begin{align*}
\frac{1}{2} \frac{d}{dt} |\bV|^2 + \ip{B(\bV,\bv_1)}{\bV} + \RE^{-1}\|\bV\|^2 = \ip{-\mu P_\sigma I_h(\bV)}{\bV}
\end{align*}
which implies that
\begin{align*}
 \frac{1}{2} \frac{d}{dt} |\bV|^2 + \RE^{-1}\|\bV\|^2 &\leq c\|\bv_1\||\bV|\|\bV\| + \mu(\sqrt{c_0}h + \lambda_1^{-1})\|\bV\||\bV| \\
 &\leq \frac{\mu^2(\sqrt{c_0}h + \lambda_1^{-1})^2}{\RE^{-1}} |\bV|^2 + \frac{\RE^{-1}}{4} \|\bV\|^2 \\
 &\phantom{=} + \frac{c^2}{2\RE^{-1}}\|\bv_1\|^2|\bV|^2 + \frac{\RE^{-1}}{2}\|\bV\|^2.
\end{align*}
Thus, 
\begin{align*}
 \frac{d}{dt} |\bV|^2 &\leq \Big( \frac{\mu^2(\sqrt{c_0}h + \lambda_1^{-1})^2}{\RE^{-1}} + \frac{c^2}{2\RE^{-1}}\|\bv_1\|^2\Big)|\bV|^2
\end{align*}
and Gr{\"o}nwall's inequality implies, for a.e. $0 \leq t \leq T$,
\begin{align*}
 |\bV(t)|^2 \leq |\bV(0)|^2 \text{exp}\Big(\int_0^T  \frac{\mu^2(\sqrt{c_0}h + \lambda_1^{-1})^2}{\RE^{-1}} + \frac{c^2}{2\RE^{-1}}\|\bv_1\|^2 dt \Big).
\end{align*}
But $\bV(0) = 0$, and thus $\|\bV\|_{L^\infty(0,T;H)} = 0$ implies that $\bV \equiv 0$.  Hence, solutions to \eqref{NSEmodifiedSens} are unique.
\end{proof}


\section{Conclusion}\label{secConclusion}

In this article, we proved well-posedness of the sensitivity equations for the 2D incompressible Navier-Stokes equations and the associated AOT data assimilation system.  Specifically, we proved the existence and uniqueness of global solutions to these equations.  A byproduct of the proof is that the sensitivity of solutions to the equations involved in the algorithm are bounded in appropriate spaces.  Hence, changing the Reynolds number, or equivalently the viscosity, mid-simulation as in \cite{Carlson_Hudson_Larios_2018} does not result in major aberrations in the solution.  We note that in the present context, our proof is somewhat non-standard, in that we proved the existence by showing that the difference quotients converge (or at least, have a subsequence that converges) to a solution of the equations.  We believe this is the first such rigorous proof that the sensitivity equations for the 2D Navier-Stokes equations are globally well-posed, although formal proofs have been given in other works, cited above.

\section*{Acknowledgements}
 \noindent
 E.C. would like to give thanks for the kind hospitality of the COSIM group at Los Alamos National Laboratory where some of this work was completed.  
 The research of E.C. was supported in part by the NSF GRFP grant no. 1610400. The research of A.L. was supported in part by the NSF grants no. DMS-1716801 and CMMI-1953346. 
 



\end{document}